 \documentclass[12pt,epsfig,amsfonts]{amsart}
\usepackage{ulem}
 \usepackage{lineno}

\usepackage{amsmath,amsthm,amssymb,mathrsfs,amscd,epsfig,color}
\usepackage{url}
\usepackage{graphicx}
\usepackage{mathrsfs}
\usepackage{ulem}
\usepackage{fancybox}
\usepackage{pb-diagram}
\usepackage{comment}
\usepackage{mathtools}
\usepackage{centernot}
\usepackage{placeins} 
\usepackage{url}

\makeatletter

\newcommand{\confrac}[2]{%
  \frac{\displaystyle{%
    \strut\hfill{#1}\hfill\;\vrule}}%
      {\displaystyle{%
       \strut\vrule\;\hfill{#2}\hfill}}}%

\makeatletter

    \newcommand\contFrac{\@ifstar{\@contFracStar}{\@contFracNoStar}}

   \def\singleContFrac#1#2{%
        \begin{array}{@{}c@{}}%
            \multicolumn{1}{c|}{#1}%
            \\%
            \hline%
           \multicolumn{1}{|c}{#2}%
        \end{array}%
   }

    \def\@contFracNoStar#1{%
        \mathchoice{
            \@contFracNoStarDisplay@#1//\@nil%
        }{
            \@contFracNoStarInline@#1//\@nil%
        }{
            \@contFracNoStarInline@#1//\@nil%
        }{
            \@contFracNoStarInline@#1//\@nil%
        }%
    }

    \def\@contFracNoStarDisplay@#1//#2\@nil{%
        \@ifmtarg{#2}{%
            #1%
        }{%
            #1+\cfrac{1}{\@contFracNoStarDisplay@#2\@nil}%
        }%
    }

        \def\@contFracNoStarInline@#1//#2\@nil{%
            \@ifmtarg{#2}{%
                #1%
            }{%
                #1 \@@contFracNoStarInline@@#2\@nil%
            }%
        }
        \def\@@contFracNoStarInline@@#1//#2\@nil{%
            \@ifmtarg{#2}{%
                + \singleContFrac{1}{#1}%
            }{%
                + \singleContFrac{1}{#1} \@@contFracNoStarInline@@#2\@nil%
            }%
        }

    \def\@contFracStar#1{%
        \mathchoice{
            \@contFracStarDisplay@#1////\@nil%
        }{
            \@contFracStarInline@#1//\@nil%
        }{
            \@contFracStarInline@#1//\@nil%
        }{
            \@contFracStarInline@#1//\@nil%
        }%
    }

    \def\@contFracStarDisplay@#1//#2//#3\@nil{%
        \@ifmtarg{#2}{%
            #1%
        }{%
            #1 + \cfrac{#2}{\@contFracStarDisplay@#3\@nil}%
        }%
    }

        \def\@contFracStarInline@#1//#2\@nil{%
            \@ifmtarg{#2}{%
                #1%
            }{%
                #1 \@@contFracStarInline@@#2\@nil%
            }%
        }
        \def\@@contFracStarInline@@#1//#2//#3\@nil{%
            \@ifmtarg{#3}{%
                + \singleContFrac{#1}{#2}%
            }{%
                + \singleContFrac{#1}{#2} \@@contFracStarInline@@#3\@nil%
            }%
        }
\makeatother

       \numberwithin{equation}{section}
\theoremstyle{plain}

\newtheorem{thm}{Theorem}[section]
\newtheorem{lem}[thm]{Lemma}

\newtheorem{pro}[thm]{Proposition}
\theoremstyle{definition}

\newtheorem{rem}[thm]{Remark}

\newtheorem*{prf*}{Proof}

\newtheorem*{pf*}{}
\newtheorem*{lem*}{LemmaA}
\newtheorem*{lm*}{LemmaB}

\title[Sets of numbers whose continued fractions contain APs]{Hausdorff dimension of sets of numbers\\
whose continued fractions contain\\ arbitrarily long arithmetic progressions}
\author{Yuto Nakajima, Hiroki Takahasi, Baowei Wang}

\date{}

\address{Faculty of Science and Engineering, Doshisha University, Kyoto, 610-0394, JAPAN}
\email{yunakaji@mail.doshisha.ac.jp}
\address{Keio Institute of Pure and Applied Sciences (KiPAS),  Department of Mathematics, Keio University, Yokohama, 223-8522, JAPAN}  \email{hiroki@math.keio.ac.jp}

\address{School of Mathematics and Statistics, Huazhong University of Science and Technology, Wuhan 430074, CHINA}
\email{bwei\_wang@hust.edu.cn}


\begin{document}

\begin{abstract}Continued fractions with
prescribed structures on sequences of
their partial quotients have been intensively studied in the literature.
 As far
as an integer sequence, especially a randomly generated one
is concerned,
an attractive question is whether it contains arbitrarily long arithmetic progressions. In this paper we study the fractal structure of irrational numbers
whose sequences of partial quotients are strictly increasing and contain arbitrarily long, quantified arithmetic progressions.

\end{abstract}

\maketitle

\section{Introduction }
It looks like a paradox that randomly generated sequences always contain given patterns almost surely. For example, almost all real numbers are normal numbers, i.e., for almost all $x$, any given word $(w_1,\ldots, w_k)$
appears in the digit sequence of the $b$-adic expansion of $x$, $b\in\mathbb N_{\geq2}$ infinitely many times with frequency $b^{-k}$. As far as a randomly generated increasing integer sequence is concerned, 
 an attractive question is whether it contains arbitrarily long arithmetic progressions (APs). 

Szemer\'edi \cite{S75} confirmed a long standing conjecture of Erd\H{o}s and Tur\'an by showing that any subset of $\mathbb N$ with positive upper density contains APs of arbitrary lengths. Furstenberg \cite{Fur77,Fur81} established the multiple recurrence theorem and the correspondence principle, and used them to derive an ergodic proof of Szemer\'edi theorem.
Generalizing Roth's harmonic analysis proof \cite{R53} of the existence of APs of length $3$ in a subset of positive integers with positive upper density, Gowers \cite{G01} obtained a new proof of Szemer\'edi's theorem with concrete bounds.
Green and Tao \cite{GT08} proved that the set of prime numbers contains arbitrarily long APs.

While the existence of APs in subsets of $\mathbb N$ has been widely studied, much less is known about their appearance 
in structured, but thin sets arising from number-theoretic expansions. 
In this paper, we connect an increasing integer sequence with the number-theoretic expansion, i.e the sequence of partial quotients in the continued fraction expansion of an irrational number. Each irrational $x$ in $(0,1)$
   is expanded into the infinite
 regular continued fraction
\[x=\confrac{1 }{a_{1}(x)} + \confrac{1 }{a_{2}(x)}+ \confrac{1 }{a_{3}(x)}  +\cdots,\]
where $a_n(x)\in\mathbb N$ for  $n\geq1$ are called the partial quotients of $x$. Define
\[J=\{x\in (0,1)\setminus\mathbb Q\colon a_n(x)<a_{n+1}(x)\ \text{ for every } n\geq1\},\]
For a randomly chosen irrational number $x\in J$, it attaches a randomly generated increasing sequence via the sequence of partial quotients in the continued fraction expansion of $x$. One would like to clarify whether it is true that for ``most'' $x\in J$, the sequence of its partial quotients contains arbitrarily long arithmetic progressions.



APs in continued fractions were first studied by Tong and Wang \cite{TW}, where they proved that the set
of $x\in J$ whose partial quotients $\{a_n(x)\colon n\in\mathbb N\}$ contains APs of arbitrary lengths and arbitrary common differences
is of Hausdorff dimension $1/2$ which equals the dimension of $J$.
Since the map $x\in J\mapsto \{a_n(x)\colon n\in\mathbb N\}\in 2^{\mathbb N}$ induces a bijection between $J$ and the collection of infinite subsets of $\mathbb N$, $J$ is indeed a natural set to investigate.
Geometrically, $J$ corresponds to the set of oriented geodesics on the modular surface $\mathbb H/{\rm SL}(2,\mathbb Z)$ that go deeper and deeper into the cusp \cite{Ser85}.
Zhang and Cao \cite{ZC2} proved that
the set of points in $x\in J$ where the sequence of partial quotients of $x$ have upper Banach density $1$ is of Hausdorff dimension $1/2$.
In a related but different direction,
the first and the second-named authors \cite{NT25}
considered a larger set
\[E=\{x\in (0,1)\setminus\mathbb Q\colon a_m(x)\neq a_{n}(x)\ \text{ for all } m,n\in\mathbb N\text{ with }m\neq n\},\]
and proved that for any set $S\subset \mathbb N$ with positive upper Banach density, there is a set $E_S\subset E$ of Hausdorff dimension $1/2$ such that for any $x\in E_S$ and any integer $\ell\geq3$, there is a strictly increasing sequence $n_1,n_2,\ldots,n_\ell$ in $\mathbb N$ such that
 $a_{n_1}(x),a_{n_2}(x),\ldots,a_{n_\ell}(x)$ is an AP contained in $S$.

The sets considered above give no quantitative information about the arithmetic progressions, and  these subsets of $J$ considered in \cite{TW,ZC2} have full Hausdorff dimension in $J$.
What are the sizes of subsets of $J$ with quantitative information on APs? Do these further restrictions cause a dimension drop?
To tackle these questions we introduce the following two sets.

Let $(\nu_n)_{n=1}^{\infty}$ and $(\sigma_n)_{n=1}^{\infty}$ be two
increasing sequences of integers in $\mathbb N$ with $\sigma_{n+1}-\sigma_n\geq n$ for all sufficiently large $n$.
Define respectively the set
\[F((\nu_n)_{n=1}^\infty)=\{x\in J\colon a_{n}(x),\ldots,a_{n+\nu_n-1}(x)\ {\text{is an AP for infinitely many}}\ n\geq1\},\]
i.e., with given length on the APs at the $n$th position; 
and the set
\[ G((\sigma_n)_{n=1}^\infty)=\{x\in J\colon a_{\sigma_n}(x),\ldots,a_{\sigma_n+n-1}(x) \ {\text{is an AP for all}}\ n \ {\text{large}}\}\]
i.e., with given positions where an AP of length $n$ follows. We call them {\it $F$-sets} and {\it $G$-sets} respectively.


The assumption on $(\sigma_n)_{n=1}^\infty$
ensures that there is no overlap between the two consecutive APs $a_{\sigma_n}(x),\ldots,a_{\sigma_n+n-1}(x)$ and  $a_{\sigma_{n+1}}(x),\ldots,a_{\sigma_{n+1}+n}(x)$.


\subsection{Statements of results}
We determine the Hausdorff dimension of $F$-sets
 and $G$-sets.
Let $\dim_{\rm H}$ denote the Hausdorff dimension on the Euclidean space $[0,1]$.

\begin{thm}
\label{thmE}
Let $(\nu_n)_{n=1}^\infty$ be an increasing sequence of integers in $\mathbb N$. We have \[\dim_{\rm H} F((\nu_n)_{n=1}^\infty)=\begin{cases}\displaystyle{\frac{1}{2(1+\alpha)}}&\text{ if }\alpha<\infty,\\
 \ \ \ 0&\text{ otherwise.}  \end{cases},\ \ \  {\text{where}}\ \alpha=\liminf_{n\to\infty}\frac{\nu_n}{n}. \]
\end{thm}

\begin{thm}\label{thmG}Let $(\sigma_n)_{n=1}^\infty$ be an increasing sequence of integers in $\mathbb{N}$
with $\sigma_{n+1}-\sigma_n\ge n$ for all $n\ge 1$ and
\begin{equation}\label{thmG-eq}\lim_{n\to\infty}\frac{\sigma_{n+1}-\sigma_n}{n}:= \beta\ge 1.\end{equation}
Then we have \[\dim_{\rm H}  G((\sigma_n)_{n=1}^\infty)=\begin{cases}\vspace{1.5mm}\displaystyle{\frac{\beta-1}{2\beta}}&\text{ if }\beta<\infty,\\
\displaystyle{\ \  {1}/{2}}&\text{ otherwise}.\end{cases}
\]
\end{thm}

Ramharter \cite{R85} proved that $\dim_{\rm H}J=1/2$. His result is
   a refinement of Good's
landmark theorem \cite[Theorem~1]{G} in dimension theory of continued fractions
which states that the set $\{x\in(0,1)\setminus\mathbb Q\colon a_n(x)\to\infty\text{ as }n\to\infty\}$ is of Hausdorff dimension $1/2$.
 As asked before: 
 what kind of prescribed conditions on APs  cause a dimension drop from $1/2$. On $F$-sets, by Theorem~\ref{thmE} there is no dimension drop if and only if $\alpha=0$. On $G$-sets, by Theorem~\ref{thmG}
there is no dimension drop if and only if $\beta=\infty$.

\subsection{Outlines of proofs of the main results}\label{outline} To prove the theorems, we establish upper and lower bounds of the Hausdorff dimension of the sets. 
The upper bounds rely on the covering argument by fundamental intervals, and the lower bounds rely on the mass distribution principle. We recall these basic ingredients in \S2. In \S3  we prove upper bounds (Propositions~\ref{upper-prop} and \ref{upper-prop-new}). In \S4 we prove lower bounds (Propositions~\ref{lower-prop} and \ref{lower-prop2}), and combine them with the upper bounds obtained in \S3 to complete the proofs of the theorems.

The afore-mentioned previous works \cite{NT25,TW,ZC2} on sets
of Hausdorff dimension $1/2$ related to APs were concerned with lower bounds only, as
the upper bound $1/2$ was known \cite{G,R85}.
For points $x$ in $F$- or $G$-sets, their partial quotients $(a_n(x))_{n=1}^\infty$ are alternate concatenations of segments of
two kinds: (i) segments consisting of arithmetic progressions, called {\it AP segments}; (ii)
segments with no constraint, called {\it free segments}.
We choose coverings that capture the distributions of these segments. The strict monotonicity of digits is crucial for
an estimation of their Hausdorff dimensions. 

Establishing the lower bounds breaks into three steps. First we construct a subset of the target set, and then construct a Borel probability measure that gives full measure to this subset. Finally we use the mass distribution principle, and obtain a lower bound of the Hausdorff dimension of the subset that is also a lower bound of the target set. To obtain the lower bound that matches the upper one obtained in \S3, we need to construct a good subset in the first step. Although the first and second steps can be in part generalized (see Lemma~\ref{abstract}),
the last step requires a detailed analysis, for both $F$-sets and $G$-sets separately, depending on the lengths and positions of APs.

\if0\textcolor{red}{\medskip
\subsection{Open questions}
We close the introduction by posing several questions.
\begin{itemize}
\item (relaxing the strict monotonicity of digits)
Can one prove a statement analogous to Theorem~\ref{thmE} or \ref{thmG} replacing $J$ by the set
\[\{x\in (0,1)\setminus\mathbb Q\colon a_n(x)\leq a_{n+1}(x)\ \text{ for every } n\geq1\}?\]
Also, 
can one replace $J$ by the set $E$ considered in \cite{NT25}?
We re-iterate that the strict monotonicity of digits is crucial for our upper bounds obtained in this paper.
\item (replacing $J$ by subsets with fast growing digits)
Can one prove a statement analogous to Theorem~\ref{thmE} or \ref{thmG} replacing $J$ by its subsets of the form
\[\{x\in J\colon a_n(x)\geq \psi(n)\ \text{ for every } n\geq1\}?\]
Here, $\psi\colon\mathbb N\to\mathbb R_+$ is a function with
$\psi(n)\to\infty$ as $n\to\infty$. For some $\psi$ that grows fast, the Hausdorff dimension drops from $1/2$, see  \cite{FWLT97,Luc97} for details.
\end{itemize}}\fi
\section{Preliminaries}
This section summarizes results that will be used in the proofs of the theorems. 
\subsection{Fundamental intervals for the continued fraction}\label{CF-sec}
For each $n\in \mathbb N$
and $(a_1,\ldots, a_n)\in \mathbb N^n$,
define non-negative integers
 $p_n$ and  $q_n=q_n(a_1,..., a_n)$ by the recursive formulas
\begin{equation}\label{pq}\begin{split}p_{-1}=1,\ p_0=0,\ p_i&=a_ip_{i-1}+p_{i-2} \ \text{ for  } i=1,\ldots,n,\\
q_{-1}=0,\ q_0=1,\ q_i&=a_iq_{i-1}+q_{i-2} \ \text{ for  } i=1,\ldots,n.\end{split}\end{equation}
For $n\in\mathbb N$ and
$(a_1,\ldots, a_n)\in \mathbb N^n$, we define an {\it $n$-th fundamental interval} by
\[ I_n(a_1,\ldots, a_n)=
\begin{cases}\vspace{1mm}
\left[\displaystyle{\frac{p_n}{q_n}}, \frac{p_n+p_{n-1}}{q_n+q_{n-1}}\right)&\text{if}\ n\ \text{is even,}\\
\displaystyle{\left(\frac{p_n+p_{n-1}}{q_n+q_{n-1}}, \frac{p_n}{q_n}\right]}&\text{if}\ n\ \text{is odd,}
 \end{cases}
 \]
This interval represents the set of elements of $(0, 1]$ that have the finite or infinite regular continued fraction expansion beginning by $a_1,\ldots,a_n$, i.e.,
\[I_n(a_1,\ldots, a_n)=\{x\in (0, 1]\colon a_i(x)=a_i\ \text{ for }i=1,\ldots,n\}.\]

Let $|\cdot|$ denote the Euclidean diameter of sets in $[0,1]$. 
\begin{lem}[{\cite{IK,Khi64,NT25}}]\label{qn}
For any $n\ge 1$ and any $(a_1,\ldots, a_n)\in \mathbb N^{n},$ we have
\begin{equation}\label{qn-1} \frac{1}{2q_n^2}\leq|I_n(a_1,\ldots,a_n)|
\leq\frac{1}{q_n^2}\end{equation}
and \begin{equation}\label{qn-2}\frac{1}{2}\prod_{i=1}^{n}\frac{1}{(a_{i}+1)^{2}}\le |I_n(a_1,\ldots, a_n)|\le\prod_{i=1}^{n}\frac{1}{a_{i}^{2}}.\end{equation}
Moreover, for any $1\le k\le n$ we have
\begin{equation}\label{qn-3}1\le \frac{q_n(a_1,..., a_n) }{q_k(a_1,..., a_k)q_{n-k}(a_{k+1},..., a_{n}) }\le 2.\end{equation}
\end{lem}

\subsection{The Gauss map}\label{gauss}
The regular continued fraction is generated by the iteration of
the Gauss map $T\colon(0,1]\to[0,1)$ given by
\[T(x)=\frac{1}{x}-\left\lfloor \frac{1}{x}\right\rfloor.\]
We have $a_n(x)=\lfloor 1/T^{n-1}x\rfloor$ for $x\in(0,1)\setminus\mathbb Q$ and $n\in\mathbb N$. In particular, $T$
acts on the set of infinite continued fractions as the shift:
$a_n(T^kx)=a_{k+n}(x)$ for $x\in(0,1)\setminus\mathbb Q$ and
$k,n\in\mathbb N$. Since the Hausdorff dimension is invariant under the action of a bi-Lipschitz map, the action of $T$ preserves the Hausdorff dimension of sets.

\subsection{Key inequalities}\label{key-up}
As outlined in \S\ref{outline},
to obtain upper bounds we choose coverings that capture distributions of AP and free segments.
The next lemma will be used to bound contributions from AP segments.
\begin{lem}\label{lem-AP} Let $s\in(0,1/2]$. For all $a, \ell\in\mathbb N$ with $2s\ell >1$ we have
\[\sum_{M=1}^{\infty}\prod_{i=0}^{\ell}\frac{1}{(a+iM)^{2s}}
\leq\frac{1}{a^{2s\ell -1}}\frac{2s\ell}{2s\ell-1}.\]
\end{lem}
\begin{proof}Since
$\prod_{i=0}^{\ell}(a+iM)\geq (a+M)^{\ell}\geq a^{\ell}+M^{\ell}$, we have
\[ \begin{split}\sum_{M=1}^{\infty}\prod_{i=0}^{\ell}\frac{1}{(a+iM)^{2s}}
&\le  \sum_{M=1}^{\infty}\frac{1}{(a^{\ell}+M^{\ell})^{2s}}\leq\sum_{M=1}^{ a }\frac{1}{a^{2s\ell }}+\sum_{M=a+1}^\infty\frac{1}{M^{2s\ell}}\\
&\le \frac{1}{a^{2s\ell -1}}+ \frac{1 }{a^{2s\ell -1}}\cdot\frac{1}{2s\ell-1}=\frac{1}{a^{2s\ell -1}}\frac{2s\ell }{2s\ell-1},\end{split}\]
as required.
 \end{proof}

The next lemma will be used in two ways: to bound contributions from one AP segment and the preceding remaining segments for $F$-sets; to bound contributions
from free segments and the subsequent AP segments for $G$-sets.

\begin{lem}\label{descend}
Let $c,n\in\mathbb N_{\geq2}$ and $s\in(0,1/2]$.
For any $\gamma\ge 2-(n-1)(2s-1)$ we have
\[\sum_{\substack{(a_1,\ldots,a_n)\in\mathbb N^n  \\ c\leq a_1<a_2<\cdots<a_n}}\frac{1}{(a_1\cdots a_{n-1})^{2s}a_n^{\gamma}}\leq \frac{1}{(c-1)^{\gamma+n(2s-1)-2s}}.\] \end{lem}
\begin{proof}
For any $\delta>1$
and any integer $t\geq2$, we have
\begin{equation}\label{lem-beta}\sum_{n=t}^\infty \frac{1}{n^{\delta}}\leq\int_{t-1}^\infty x^{-\delta}dx=\frac{1}{(t-1)^{\delta-1}}\frac{1}{\delta-1}.\end{equation}
We will use this inequality repeatedly. Since $0<s\le 1/2$ and by the condition on $\gamma$, it follows that for all $0\le i\le n-1$,
\begin{equation}\label{1}\gamma+i(2s-1)\ge \gamma+(n-1)(2s-1)\ge 2.\end{equation}
Thus we have
\[\begin{split}
&\sum_{\substack{(a_1,\ldots,a_n)\in\mathbb N^n  \\ c\leq a_1<a_2<\cdots<a_n}}\frac{1}{(a_1\cdots a_{n-1})^{2s}a_n^{\gamma}}\\
=&\sum_{\substack{(a_1,\ldots,a_{n-1})\in\mathbb N^{n-1}  \\ c\leq a_1<a_2<\cdots<a_{n-1}}}\frac{1}{(a_1\cdots a_{n-2})^{2s}}\frac{1}{a_{n-1}^{2s}}\sum_{a_n=a_{n-1}+1}^\infty\frac{1}{a_n^{\gamma}}\\
\leq&\sum_{\substack{(a_1,\ldots,a_{n-1})\in\mathbb N^{n-1}  \\ c\leq a_1<a_2<\cdots<a_{n-1}}}\frac{1}{(a_1\cdots a_{n-2})^{2s}}\frac{1}{a_{n-1}^{\gamma+2s-1}}\frac{1}{\gamma-1}\\
=&\sum_{\substack{(a_1,\ldots,a_{n-2})\in\mathbb N^{n-2}  \\ c\leq a_1<a_2<\cdots<a_{n-2}}}\frac{1}{(a_1\cdots a_{n-3})^{2s}}
\times\frac{1}{a_{n-2}^{2s}}\sum_{a_{n-1}=a_{n-2}+1}^\infty\frac{1}{a_{n-1}^{\gamma+2s-1}}\frac{1}{\gamma-1}\\
\leq&\sum_{\substack{(a_1,\ldots,a_{n-2})\in\mathbb N^{n-2}  \\ c\leq a_1<a_2<\cdots<a_{n-2}}}\frac{1}{(a_1\cdots a_{n-3})^{2s}}\times\frac{1}{a_{n-2}^{\gamma+2(2s-1)}}\frac{1}{\gamma-1}\frac{1}{\gamma+2s-2}.\end{split}\]
Iterating this argument and by (\ref{1}), we arrive at the last step: \[\begin{split}\sum_{\substack{(a_1,\ldots,a_n)\in\mathbb N^n  \\
c\leq a_1<a_2<\cdots<a_n}}\frac{1}{(a_1\cdots a_{n-1})^{2s}a_n^{\gamma}}&\leq\sum_{a_1=c }^\infty\frac{1}{a_{1}^{\gamma+(n-1)(2s-1)}}
\prod_{i=0}^{n-2}\frac{1}{\gamma-1+i(2s-1)}.\end{split}\]
Noticing that $\gamma-1+i(2s-1)\ge 1$ for $0\leq i\leq n-1$ and using \eqref{lem-beta} again, one gets
\begin{align*}\sum_{\substack{(a_1,\ldots,a_n)\in\mathbb N^n  \\
c\leq a_1<a_2<\cdots<a_n}}\frac{1}{(a_1\cdots a_{n-1})^{2s}a_n^{\gamma}}&\leq \sum_{a_1=c }^\infty\frac{1}{a_{1}^{\gamma+(n-1)(2s-1)}}
\leq \frac{1}{(c-1)^{\gamma+n(2s-1)-2s}},
\end{align*}
as required.
\end{proof}

\subsection{Mass distribution principle}\label{mass-sec}
For $x\in[0,1]$ and $r>0$,
let $B(x,r)=\{y\in [0,1]\colon|x-y|\leq r\}$.
We cite the well known {\it the mass distribution principle}.
\begin{lem}[\cite{Fal97, Fal14}]\label{mass1}
Let $\Lambda$ be a Borel subset of $[0,1]$ and let $\mu$ be a Borel probability measure on $[0,1]$ with $\mu(\Lambda)=1$. If there exists $K>0$ such that for any $x\in \Lambda$,
\[\liminf_{r\to0}\frac{\log\mu(B(x,r))}{\log r}\geq K,\] then $\dim_{\rm H}\Lambda\geq K$.
\end{lem}

\subsection{Construction of a Borel probability measure}\label{const-sec}
To obtain lower bounds on Hausdorff dimension,
we will apply Lemma~\ref{mass1} to Borel probability measures on $[0,1]$ and Borel subsets of $J$ given in the next lemma.
\begin{lem}\label{abstract}
Let $\{V_1,W_1,V_2,W_2,\ldots\}$ be a partition of $\mathbb N$ into bounded intervals
such that
 $\max V_k+1=\min W_k$ and $\max W_k+1=\min V_{k+1}$
for all $k\geq1$.
Let
$(L_n)_{n=1}^\infty$ be a sequence of
non-empty bounded intervals in $\mathbb N$, and
let
\[\Lambda=
\left\{
\begin{tabular}{l}
\!\!\!$x\in (0,1)\setminus\mathbb Q\colon$\!\!\!
 $a_n(x)\in L_n$ for $n\in V_k$, $k\in \mathbb N$ and \!\!\!\!\\
 \ \ \ \ \ \ \ \     $a_{n}(x)=a_{\max V_k }(x)+n-\max V_k$ for $n\in W_k$, $k\in \mathbb N$ \end{tabular}\!\!
\right\}.
\]
\begin{itemize}\item[(a)] There exists a Borel probability measure $\mu$
on $[0,1]$ such that:

\begin{itemize}
\item $\mu(\Lambda)=1$;
\item for all $n\in\mathbb N$ and all $(a_1,\ldots,a_n)\in\mathbb N^n$ with
$I_n(a_1,\ldots,a_n)\cap \Lambda\neq\emptyset$,
\[\mu(I_n(a_1,\ldots, a_n))=\begin{cases}\displaystyle{\prod_{i=1}^n
\frac{1}{    \#L_i  }  }&\text{ for }n\in V_1,\\
\displaystyle{\prod_{m=1}^k
\prod_{i\in V_m}\frac{1}{\#L_i   }}&\text{ for }
n\in W_k,\
k\geq1,\\
\displaystyle{\begin{split}&\prod_{m=1}^k\prod_{i\in V_m}\frac{1}{\#L_i  }\prod_{i=\min V_{k+1} }^{n}\frac{1}{\#L_i }\end{split}}
&\text{ for }
n\in V_{k+1},\ k\geq1.
\end{cases}\]
\end{itemize}
\item[(b)] If $L_n=[(2n)^t,(2n+1)^t)$ for all $n\geq1$ with $t\in\mathbb N_{\geq2}$, then $\Lambda\subset J$.
\end{itemize}
\end{lem}
\begin{proof}
Taking a sequence $(\mu_n)_{n=1}^\infty$ of Borel probability measures on $[0,1]$
such that
for all $n\in\mathbb N$ and all $(a_1,\ldots,a_n)\in\mathbb N^n$ with
$I_n(a_1,\ldots,a_n)\cap \Lambda\neq\emptyset$, we set
\[\mu_n(I_n(a_1,\ldots, a_n))=\begin{cases}\displaystyle{\prod_{i=1}^n
\frac{1}{    \#L_i  }  }&\text{ for }n\in V_1,\\
\displaystyle{\prod_{m=1}^k
\prod_{i\in V_m}\frac{1}{\#L_i   }}&\text{ for }
n\in W_k,\
k\geq1,\\
\displaystyle{\begin{split}&\prod_{m=1}^k\prod_{i\in V_m}\frac{1}{\#L_i  }\prod_{i=\min V_{k+1} }^{n}\frac{1}{\#L_i }\end{split}}
&\text{ for }
n\in V_{k+1},\ k\geq1.
\end{cases}\]
By definition, for all $n,q\in\mathbb N$ with $q>n$ and all $(a_1,\ldots,a_n)\in\mathbb N^n$, we have
\begin{equation}\label{compatible}\mu_n(I_n(a_1,\ldots,a_n))=\mu_q(I_n(a_1,\ldots,a_n)).\end{equation}
Each $\mu_n$ can be viewed as a Borel probability measure on $\mathbb N^n$,
and \eqref{compatible} can be viewed as
Kolmogorov's compatibility condition on the sequence $(\mu_n)_{n=1}^\infty$.
By  Kolmogorov's extension theorem,
 there is a unique Borel probability measure $\mu$ on $[0,1]$ such that for all $n\geq1$ and all $(a_1,\ldots,a_n)\in\mathbb N^n$,
 \[\mu(I_n(a_1,\ldots, a_n))=\mu_n(I_n(a_1,\ldots, a_n)).\] Hence
 the formula in (a) holds. If $n\geq1$ and $I_n(a_1,\ldots,a_n)\cap \Lambda=\emptyset$ then $\mu_n(I_n(a_1,\ldots, a_n))=0$. This yields $\mu(\Lambda)=1$ as required in (a).

 To prove (b), it suffices to show that $(a_n(x))_{n=1}^\infty$ is strictly increasing for any $x\in \Lambda$.
 Since $\max L_n<\min L_{n+1}$ for all $n\geq1$,
it suffices to show that \[a_{\min W_k}(x)+\#W_k-1< a_{ \min V_{k+1}}(x)\ \text{ for all }k\geq1.\]
Note that $\max V_k+1=\min W_k$ and $\min W_k+\#W_k=\min V_{k+1}$. 
Since$$a_{\min W_k}(x)=a_{\max V_k}+1<(2\max V_k+1)^t+1, \ \ \ (2\min V_{k+1} )^t\leq a_{\min V_{k+1} }(x),$$
 we obtain
\begin{equation*}\begin{split}a_{\min W_k}(x)+\# W_k-1&<(2\min W_k-1)^t+\# W_k-1\\ &\leq(2\min W_k+2\# W_k)^t\\
&= (2\min V_{k+1})^t \leq
a_{\min V_{k+1}}(x)\end{split}\end{equation*}
as required.
\end{proof}

\section{Upper bounds on Hausdorff dimension}
In this section we establish upper bounds on the Hausdorff dimension of $F$-sets and $G$-sets.
\subsection{Upper bounds for $F$-sets }\label{upper-sec}
Recall that \[F((\nu_n)_{n=1}^\infty)=\{x\in J\colon a_{n}(x),\ldots,a_{n+\nu_n-1}(x)\ {\text{is an AP for infinitely many}}\ n\in \mathbb{N}\}.\]
The upper bound for $F$-sets follows from the next statement.

\begin{pro}\label{upper-prop}
Let $(\nu_n)_{n=1}^\infty$ be an increasing sequence of  integers in $\mathbb N$.
 Then we have \[\dim_{\rm H} F((\nu_n)_{n=1}^\infty)\leq\frac{1}{2(1+\alpha)}, \ {\text{where}} \ \alpha=\liminf_{n\to\infty}\frac{\nu_n}{n}.\]
 \end{pro}
\begin{proof}
Since the Gauss map $T$ is locally bi-Lispchtiz, one has $\dim_{\rm H} TE=\dim_{\rm H} E$ for any subset $E$ in $[0,1].$ Thus, we can pose extra condition on the first partial quotients on the points in $F((\nu_n)_{n=1}^\infty)$ without changing its dimension.

More precisely, for all $i\in\mathbb N$ we have
\[T^i\left(F((\nu_n)_{n=1}^\infty)
\right)\subset F((\nu_{i+n})_{n=1}^\infty)\cap\bigcup_{a_1=i+1}^\infty I_1(a_1),\]
where the second intersection is because the partial quotients of $x\in F((\nu_n)_{n=1}^\infty)$ is strictly increasing.
Then up to shifting the indices it suffices to show
\[\dim_{\rm H} F'\le \frac{1}{2(1+\alpha)},\] where
\[F'=F((\nu_n)_{n=1}^\infty)\cap \bigcup_{a_1=3}^\infty I_1(a_1).\]

Let $s>1/(2(1+\alpha))$. The definition of $\alpha$ implies that for all $n$ large enough, $$
2s\cdot \frac{\nu_n+n}{n}>1+\epsilon, \ {\text{for some}}\ \epsilon>0.
$$
Thus we can take
 $\delta\in(0,2s(\alpha+1)-1)$ small
 such that for some $N>1/\delta$,
\begin{equation}\label{ell-eq}2s(\nu_n-1) +(2s-1)(n-1)-2\geq\max\{\delta n,2\}\ \text{ for all }n\geq N.\end{equation}
For integers $n\geq3$ and $M\geq1$,
let \[D^n_M=\{(b_{1},\ldots, b_{n})\in\mathbb N^{n}\colon b_{i+1}-b_i=M\ \text{ for }i=1,\ldots, n-1 \},\]
i.e., APs with length $n$ and common difference $M$.

Consider the covering of $F'$ by fundamental intervals that captures AP segments in the sequences of the partial quotients of points in $F'$:  \[F'\subset\bigcup_{n\geq N  }\bigcup_{\substack{(a_1,\ldots, a_{n-1})\in\mathbb N^{n-1}\\ 3\leq a_1<\cdots<a_{n-1}}}\bigcup_{M=1}^{\infty}
\bigcup_{\substack{(a_n,\ldots,
a_{n+\nu_n-1})\in D^{\nu_n}_M \\ a_{n-1}<a_n }}I_{n+\nu_n-1}
(a_1,\ldots,a_{n-1},a_n,\ldots,a_{n+\nu_n-1}).\]
Using the second inequality in \eqref{qn-2}, we have
\[\begin{split}&\sum_{\substack{(a_1,\ldots, a_{n-1})\in\mathbb N^{n-1}\\ 3\leq a_1<\cdots<a_{n-1}}}\sum_{M=1}^{\infty}
\sum_{
\substack{(a_n,\ldots,
a_{n+\nu_n-1})\in D^{\nu_n}_ M \\ a_{n-1}<a_n } }|I_{n+\nu_n-1}(a_1,\ldots,a_{n-1},a_{n},\ldots,a_{n+\nu_n-1})|^s\\
&\leq\!\!\!\!\!\!\!\sum_{\substack{(a_1,\ldots, a_{n-1})\in\mathbb N^{n-1}\\ 3\leq a_1<\cdots<a_{n-1}}}
\frac{1}{(a_1\cdots a_{n-1})^{2s}}
\sum_{a_{n}= a_{n-1}+1  }^\infty\sum_{M=1}^{\infty} \left[\prod_{i=0}^{\nu_n-1}(a_n+iM) \right]^{-2s}\!\!\!\!\!\!\!.
\end{split}\]
On the contribution from the AP segments from position $n$ to $n+\nu_n-1$,
since $\nu_n\geq3$ and $2s(\nu_n-1)\geq2$ from \eqref{ell-eq},   Lemma \ref{lem-AP} gives \begin{equation}\label{upper-2a}\sum_{M=1}^{\infty}\left[\prod_{i=0}^{\nu_n-1}(a_n+iM) \right]^{-2s}
\le \frac{1}{a_{n}^{2s(\nu_n-1) -1}}\frac{2s(\nu_n-1)}{2s(\nu_n-1)-1}\le \frac{2}{a_{n}^{2s(\nu_n-1) -1}}.\end{equation}
Let $c=3$ and $\gamma=2s(\nu_n-1)-1$ {\color{red}}which satisfies the condition required in Lemma \ref{descend} by \eqref{ell-eq}. Thus plugging the inequality (\ref{upper-2a}) into the previous inequality, and then applying Lemma~\ref{descend} we get
\[
\begin{split}
\sum_{\substack{(a_1,\ldots, a_{n-1})\in\mathbb N^{n-1}\\ 3\leq a_1<\cdots<a_{n-1}}}&\sum_{M=1}^{\infty}
\sum_{
\substack{(a_n,\ldots,
a_{n+\nu_n-1})\in D^{\nu_n}_M \\ a_{n-1}<a_n } }|I_{n+\nu_n-1}(a_1,\ldots,a_{n-1},a_{n},\ldots,a_{n+\nu_n-1})|^s\\
&\leq \!\!\!\!\!\!\!\sum_{ \substack{(a_1,\ldots, a_{n-1}, a_n)\in\mathbb N^{n}\\ 3\leq a_1<\cdots<a_{n-1}<a_n}  }
\frac{2}{(a_1\cdots a_{n-1})^{2s}a_{n}^{2s(\nu_n-1) -1}}\leq \frac{2}{2^{2s(\nu_n-1) +(2s-1)(n-1)-2}}.
\end{split}
\]
Since
 $2s(\nu_n-1) +(2s-1)(n-1)-2\geq\delta n$ by
 \eqref{ell-eq}, we have
\[\frac{2}{2^{2s(\nu_n-1)+(2s-1)(n-1)-2}}\le \frac{2}{2^{\delta n}},\]
and hence
\[\begin{split}\sum_{n= N}^\infty\sum_{\substack{(a_1,\ldots, a_{n-1})\in\mathbb N^n\\ 3\leq a_1<\cdots<a_{n-1}}}\sum_{M=1}^{\infty}
&\sum_{   \substack{(a_n,\ldots,
a_{n+\nu_n-1})\in D^{\nu_n}_M \\ a_{n-1}<a_n } }|I_{n+\nu_n-1}(a_1,\ldots,a_{n-1},a_{n},\ldots,a_{n+\nu_n-1})|^s\\
&\leq \sum_{n=N}^{\infty}\frac{2}{2^{\delta n}}<\infty.
\end{split}\]
Since the supremum of the diameters of the fundamental intervals in the covering of $F'$ converges to $0$ as $N\to\infty$,
we obtain $\dim_{\rm H} F'\le s.$
By the arbitrariness of $s>1/(2(1+\alpha))$, we obtain
$\dim_{\rm H} F'\le 1/(2(1+\alpha))$ as required.\end{proof}

\subsection{Upper bounds for $G$-sets}\label{upper-G} Recall that
\[ G((\sigma_n)_{n=1}^\infty)=\{x\in J\colon a_{\sigma_n}(x),\ldots,a_{\sigma_n+n-1}(x) \ {\text{is an AP for all}}\ n \ {\text{large}}\}.\]
The upper bound for $G$-sets follows from the next statement.
A proof is similar in spirit to that of Proposition~\ref{upper-prop}, but technically more involved. We will consider coverings that capture alternate concatenations of free and AP segments in the
sequences of partial quotients of points in $G$-sets.
\begin{pro}\label{upper-prop-new}
Let $(\sigma_n)_{n=1}^\infty$ be a strictly increasing sequence of positive integers with $\sigma_{n+1}\ge \sigma_n+n$ for all $n\ge 1$ and $\lim_{n\to\infty}\frac{\sigma_{n+1}-\sigma_n}{n}:=\beta$
for some $\beta\ge 1.$
 Then we have \[\dim_{\rm H} G((\sigma_n)_{n=1}^\infty)\leq\frac{\beta-1}{2\beta }. \]
 \end{pro}

 \begin{proof} 
Let $(\beta-1)/(2\beta)<s<1/2$. By the definition of $\beta$, we have
\[\liminf_{n\to\infty}\frac{2s(\sigma_{n+1}-\sigma_{n})-(\sigma_{n+1}-\sigma_{n}-n)}{n}>0.\]
Let $\delta>1$. Choose $n_0\in \mathbb{N}$ large such that $sn\ge 2$ and
\begin{equation}\label{eq-c}(2s-1)(\sigma_{n+1}-\sigma_{n})+n-2\geq\delta, \  \ \text{ for }n\ge n_0.\end{equation}

 It is clear that \begin{align*}
   G((\sigma_n)_{n=1}^\infty)&\subset \bigcup_{j\ge n_0}\bigcup_{a_1<\cdots<a_{\sigma_j+j-1}}\{x\in I_{\sigma_j+j-1}(a_1,\cdots, a_{\sigma_j+j-1})\cap J: \\ & \qquad \qquad \qquad \qquad a_{\sigma_n}(x),\ldots,a_{\sigma_n+n-1}(x)
    \ {\text{is an AP for all}}\ n>j\}\\
    &:=\bigcup_{j\ge n_0}\bigcup_{a_1<\cdots<a_{\sigma_j+j-1}}G_j.
 \end{align*}  
 By the countable stability of Hausdorff dimension,
  it suffices to show that \begin{equation}\label{reduced}\dim_{\rm H} G_j\leq s \ \text{ for all }j\geq n_0.\end{equation}
  Since $s>(\beta-1)/(2\beta)$ is arbitrary, this will imply the desired upper bound.

As in the proof of Proposition~\ref{upper-prop},
for integers $n\geq n_0$ and $M\geq1$
let \[D^n_M=\{(b_{1},\ldots, b_{n})\in\mathbb N^{n}\colon b_{i+1}-b_i=M\ \text{ for }i=1,\ldots, n-1 \}.\]
For $n\in\mathbb N$ define
\[\mathbb N^n_{>}=\{(a_1,\ldots,a_n)\in\mathbb N^n\colon a_1<\cdots<a_n\}.\]
For an integer $n\geq n_0$, put
\[\ell_n=\sigma_n-(\sigma_{n-1}+n-1)\ge 0,\]
which is the gap between two APs in the digit sequence of $x\in G((\sigma_n)_{n=1}^\infty)$.

Recall the positions of the free segments and AP segments in $G_j$: for all $n>j$
 $$(a_{\sigma_{n-1}+n-1 },\ldots,a_{\sigma_{n}-1})\ {\text{free}}; \ \ (a_{\sigma_{n}},\cdots, a_{\sigma_{n}+n-1})={\text{AP}}.$$
 Consider the covering of $G_j$ by $(\sigma_{n+1}+n)$-th fundamental intervals
$I_{\sigma_{n+1}+n}(a_1,\ldots,a_{\sigma_{n+1}+n})$ such that
$(a_{\sigma_k},\ldots,
a_{\sigma_k+k-1})\in D^k_{M_k}$, $M_k\geq1$ for $j< k\leq n+1$:
\[\begin{split}G_j\subset &\bigcup_{\substack{ (a_{\sigma_{j}+j },\ldots,a_{\sigma_{j+1}-1})\in\mathbb N^{\ell_{j+1} }_{>} \\ a_{\sigma_{j}+j }>a_{\sigma_{j}+j-1 }  } }  \bigcup_{M_{j+1}=1}^{\infty}\bigcup_{\substack{
(a_{\sigma_{j+1}},\dots,
a_{\sigma_{j+1}+j})\in D^{j+1}_{M_{j+1}}\\ a_{\sigma_{j+1}}>a_{\sigma_{j+1}-1}}} \cdots\cdots
\\
&
\bigcup_{\substack{
(a_{\sigma_n+n},\ldots,a_{\sigma_{n+1}-1})\in\mathbb N_{ >}^{\ell_{n+1} } \\ a_{\sigma_n+n}>a_{\sigma_n+n-1}  } }
\bigcup_{M_{n+1}=1 }^{\infty}\bigcup_{\substack{
 (a_{\sigma_{n+1}},\ldots,
a_{\sigma_{n+1}+n })\in D^{n+1}_{M_{n+1}} \\ a_{\sigma_{n+1}}>a_{\sigma_{n+1}-1}}  }I_{\sigma_{n+1}+n}(a_1,\ldots,a_{\sigma_{n+1}+n}).\end{split}\]
Correspondingly
we set
\[\begin{split}\mathscr{H}_n^s&=\sum_{\substack{ (a_{\sigma_{j}+j},\ldots,a_{\sigma_{j+1}-1})\in\mathbb N^{\ell_{j+1} }_{>}\\a_{\sigma_{j}+j }>a_{\sigma_{j}+j-1}}
    }\sum_{M_{j+1}=1}^{\infty}\sum_{a_{\sigma_{j+1}}>a_{\sigma_{j+1}-1}}\sum_{
(a_{\sigma_{j+1}},\ldots,
a_{\sigma_{j+1}+j})\in D^{j+1}_{M_{j+1}}}
\cdots  \cdots
 \\
&\times
\sum_{\substack{ (a_{\sigma_{n}+n},\ldots,a_{\sigma_{n+1}-1})\in\mathbb N^{\ell_{n+1} }_{>}\\a_{\sigma_{n}+n }>a_{\sigma_{n}+n-1}}
    }\sum_{M_{n+1}=1}^{\infty}\sum_{a_{\sigma_{n+1}}>a_{\sigma_{n+1}-1}}\sum_{
(a_{\sigma_{n+1}},\ldots,
a_{\sigma_{n+1}+n})\in D^{n+1}_{M_{n+1}}}
\frac{1}{(a_1\cdots a_{\sigma_{n+1}+n})^{2s}}.
\end{split}\]
From the length estimation of a cylinder \eqref{qn-2},
we will obtain $\dim_{\rm H}G_j\leq s$ once one can show $\lim\inf_{n\to \infty}\mathscr{H}_n^s$ is finite.

 We estimate $\mathscr{H}_n^s$ from above. At first, since $sn\ge 2$ for all $n\ge n_0$, one has 
   \[\frac{1}{sn-1}\frac{sn}{2sn-1}<1, \ {\text{for all}}\ n\ge n_0.\]
For each $n\ge j$, write 
\[\begin{split}
I:=&\!\!\!\!\!\!\!\sum_{\substack{ (a_{\sigma_{n}+n},\ldots,a_{\sigma_{n+1}-1})\in\mathbb N^{\ell_{n+1} }_{>}\\a_{\sigma_{n}+n }>a_{\sigma_{n}+n-1}}
    }\sum_{
a_{\sigma_{n+1}}=a_{\sigma_{n+1}-1}+1
}^\infty\sum_{M=1 }^{\infty}\\
 &\qquad \qquad\times\!\!\!\!\!\!\!\sum_{
 (a_{\sigma_{n+1}},\ldots,
a_{\sigma_{n+1}+n})\in D^{n+1}_{M}   }\frac{1}{(a_{\sigma_n+n}\cdots a_{\sigma_{n+1}-1}a_{\sigma_{n+1}}\cdots a_{\sigma_{n+1}+n})^{2s}}\\
\leq& \sum_{\substack{ (a_{\sigma_{n}+n},\ldots,a_{\sigma_{n+1}-1})\in\mathbb N^{\ell_{n+1} }_{>}\\a_{\sigma_{n}+n }>a_{\sigma_{n}+n-1}}
    }\frac{1}{
(a_{\sigma_{n}+n}
\cdots a_{\sigma_{n+1}-1})^{2s}}
\sum_{a_{\sigma_{n+1}}= a_{\sigma_{n+1}-1}+1 }^\infty\sum_{M=1 }^{\infty}\left[\prod_{i=0}^{n}(a_{\sigma_{n+1}}+iM) \right]^{-2s},\end{split}\]
by noticing that there is only vector in $D_M^n$ once the first coordinate of the vector is given. Then using Lemma~\ref{lem-AP}, we get
\begin{align}I\leq & \sum_{\substack{ (a_{\sigma_{n}+n},\ldots,a_{\sigma_{n+1}-1})\in\mathbb N^{\ell_{n+1} }_{>}\\a_{\sigma_{n}+n }>a_{\sigma_{n}+n-1}}
    }\frac{1}{
(a_{\sigma_{n}+n}
\cdots a_{\sigma_{n+1}-1})^{2s}}
\sum_{a_{\sigma_{n+1}}= a_{\sigma_{n+1}-1}+1 }^\infty\frac{1}{a_{\sigma_{n+1}}^{2sn-1}}\frac{2sn}{2sn-1}\nonumber\\
\leq&\frac{1}{2sn-2}\cdot\frac{2sn}{2sn-1}\sum_{\substack{ (a_{\sigma_{n}+n},\ldots,a_{\sigma_{n+1}-1})\in\mathbb N^{\ell_{n+1} }_{>}\\a_{\sigma_{n}+n }>a_{\sigma_{n}+n-1}}
    }\frac{1}{
(a_{\sigma_n+n }
\cdots a_{\sigma_{n+1}-1})^{2s}}\frac{1}{a_{\sigma_{n+1}-1 }^{2sn-2}} \nonumber\\
\le &\!\!\!\!\!\!\!\sum_{\substack{ (a_{\sigma_{n}+n},\ldots,a_{\sigma_{n+1}-1})\in\mathbb N^{\ell_{n+1} }_{>}\\a_{\sigma_{n}+n }>a_{\sigma_{n}+n-1}}
    }\frac{1}{
(a_{\sigma_n+n }
\cdots a_{\sigma_{n+1}-2})^{2s}}\frac{1}{a_{\sigma_{n+1}-1 }^{2sn+2s-2}}.\label{2}\end{align}
If $\ell_{n+1}=0$, i.e. $\sigma_{n+1}=\sigma_n+n$, the first summation is empty, thus it is trivial that $$I\le \frac{1}{a_{\sigma_{n+1}-1 }^{2sn+2s-2}}\le 1.$$ If $\ell_{n+1}\ge 1$, we will apply Lemma~\ref{descend} to (\ref{2}). For the current case, the condition on $\gamma$ in Lemma~\ref{descend} reads as
$$2sn+2s-2>2-(\sigma_{n+1}-\sigma_n-n-1)(2s-1)$$
which is nothing but the inequality in \eqref{eq-c}. Moreover $a_{\sigma_n+n}>n$, thus 
one arrives at
\[\begin{split}I\leq&\frac{1}{n^{2sn-2+(\sigma_{n+1}-\sigma_n-n)(2s-1) }}
 =\frac{1}{n^{(\sigma_{n+1}-\sigma_n)(2s-1)+n-2 }}\leq 
 \frac{1}{n^\delta}<1.\end{split}\] 

Hence, $\mathscr{H}_{n}^s\leq \mathscr{H}_{n-1}^{s}$ for all $n>j$. By an iteration, one has$$
\mathscr{H}_{n}^s\le \frac{1}{(a_1\cdots a_{\sigma_{j}+j-1})^{2s}}\le 1.
$$
Thus $\dim_{\rm H}G_j\leq s$ as required in \eqref{reduced}.
\end{proof}

\section{Lower bounds and proofs of the main results}
In this section we establish lower bounds on the Hausdorff dimension of $F$-sets and $G$-sets.
Then we combine them with the upper bounds obtained in \S3, and complete the proofs of the main results.

\subsection{Lower bound for $F$-sets}\label{lower-F}
The next proposition implies the lower bound of Hausdorff dimension of $F$-sets.
\begin{pro}\label{lower-prop}
Let $(\nu_n)_{n=1}^\infty$ be an increasing sequence of integers. 
Then we have \[\dim_{\rm H} F((\nu_n)_{n=1}^\infty)\geq\frac{1}{2(1+\alpha)}, \ {\text{where}}\ \alpha=\liminf_{n\to\infty}\frac{\nu_n}{n}.\]\end{pro}

\begin{proof}
Write $F=F((\nu_n)_{n=1}^\infty)$. By the definition of $\alpha$, choose a strictly increasing sequence of positive integers $(n_k)_{k=1}^\infty$  such that
\[\lim_{k\to \infty}\frac{\nu_{n_k}}{n_k}=\alpha\ \text{ and }\ n_k+\nu_{n_k}<   n_{k+1} \ \text{ for every }k\geq1,\]
and \[\text{for each $k\geq2$, $n_k$ is sufficiently large compared to $n_1,\ldots,n_{k-1}$.}\]
 To
avoid triple subscripts,
put $\nu(k)=\nu_{n_k}$.
Set $n_0+\nu(0)=1$, and for each $k\geq1$ define sets $V_k$, $W_k$ of positive integers by
\[V_k=\{n_{k-1}+\nu(k-1),\ldots,n_k\}\ \text{ and }\ W_k=\{n_k+1,\ldots,n_k+\nu(k)-1\},\]
which correspond to the positions for the free segments and AP segments respectively in the sequence of the partial quotients of $x\in F((\nu_n)_{n=1}^\infty)$.
Then $\{V_1,W_1,V_2,W_2,\ldots\}$ is a partition of $\mathbb N$ into bounded intervals.
 For an integer $t\geq2$, define 

\[\Lambda_t(F)=
\left\{
\begin{tabular}{l}
\!\!\!$x\in (0,1)\setminus\mathbb Q\colon$\!\!\!
 $(2n)^t\le a_n(x)< (2n+1)^t$ for $n\in \bigcup_{k=1}^\infty V_k$ and \!\!\!\!\\
 \ \ \ \ \ \ \ \ \ \ \ \ \ \ \ \ \ \  $a_{n}(x)=a_{n_k }(x)+n-n_k$ for $n\in W_k$, $k\in \mathbb N$ \end{tabular}
\right\}.
\]
By Lemma~\ref{abstract}(a), there exists
a Borel probability measure $\mu$ on $[0,1]$
such that $\mu(\Lambda_t(F))=1$, and
for all $n\in\mathbb N$ and all $(a_1,\ldots,a_n)\in\mathbb N^n$ with
$I_n(a_1,\ldots,a_n)\cap \Lambda_t(F)\neq\emptyset$,
\[\mu(I_n(a_1,\ldots, a_n))=\begin{cases}\displaystyle{\prod_{i=1}^n\frac{1}{(2i+1)^t-(2i)^t}}&\text{ for }n\in V_1,\\
\displaystyle{\prod_{m=1}^k
\prod_{i\in V_m}\frac{1}{(2i+1)^t-(2i)^t}}&\text{ for }
n\in W_k,\
k\geq1,\\
\displaystyle{\begin{split}&\prod_{m=1}^k\prod_{i\in V_m}\frac{1}{(2i+1)^t-(2i)^t}\\
&\times\prod_{i=\min V_{k+1} }^{n}\frac{1}{(2i+1)^t-(2i)^t}\end{split}}
&\text{ for }
n\in V_{k+1},\ k\geq1.
\end{cases}\]
By Lemma~\ref{abstract}(b), $\Lambda_t(F)$ is contained in $F$.

For each integer $t\geq2$,
we estimate the Hausdorff dimension of $\Lambda_t(F)$ from below.
For each $k\in\mathbb N$, define two positive numbers. The first one is \[A_k=\frac{\displaystyle{\sum_{m=1}^k\sum_{i\in V_m    }\log [(2i+1)^t-(2i)^t]}}{2\left[\displaystyle{\sum_{m=1}^k
\sum_{i\in V_m
    }t\log(2i+1)+\sum_{m=1}^k(\nu(m)-1)\log \Big((2n_m+1)^t+\nu(m)\Big)}\right]+\log 2}\]
where the numerator corresponds to the measure of a cylinder $I_n(a_1,\cdots, a_n)$ for all $n_k\le n<n_k+\nu(k)$ and the denominator corresponds to the length of the  cylinder $I_{n_k+\nu(k)-1}(a_1,\cdots,a_{n_k+\nu(k)-1})$. The second number is
 \[B_k=\inf_{n\in V_{k+1} }\frac{\displaystyle{\sum_{i=\min V_{k+1} }^n \log [(2i+1)^t-(2i)^t]}}{\displaystyle{2 \sum_{i=\min V_{k+1} }^{n} t\log (2i+1) }}.\]
 \begin{lem}\label{AkBk}
We have
\[\lim_{k\to\infty}A_k=  \frac{t-1}{2t(1+\alpha)} \ \text{ and } \ \lim_{k\to\infty}B_k=
 \frac{t-1}{2t}.\]\end{lem}
 \begin{proof}
Keeping the condition $\lim_{k\to\infty}\nu(k)/n_k=\alpha$, if necessary we replace each $n_k$ inductively by a sufficiently large integer compared to $n_1,\ldots,n_{k-1}$ so that the following three fractions
\[\begin{split}
&\frac{\displaystyle{\sum_{m=1}^k\sum_{i\in V_m}\log (2i)}}{n_k\log n_k},\
\frac{\displaystyle{\sum_{m=1}^k\sum_{i\in V_m   }\log (2i+1)}}{n_k\log n_k},\ \frac{\displaystyle{\sum_{m=1}^k(\nu(m)-1)\log [(2n_m+1)^t+\nu(m)]}}{t\nu(k)\log n_k}
\end{split} \]
converge to $1$
as $k\to \infty$.
This yields
$A_k\to (t-1)/(2t(1+\alpha))$ as required.
The convergence $B_k\to
 (t-1)/(2t)$ is obvious. 
 \end{proof}

\begin{lem}\label{ratio}
For any integer $t\geq2$,
 all $x\in \Lambda_t(F)$, all
  $k\in\mathbb N$ and all $n_{k}< n\le n_{k+1}$, we have
\[\frac{\log \mu(I_n(a_1(x),\ldots, a_n(x)))}{\log |I_{n}(a_1(x),\ldots, a_{n}(x))|}\ge \min\{A_k, B_k\}.\]\end{lem}
\begin{proof}
 Let $x\in \Lambda_t(F)$ and $k\in\mathbb N$.
We write $a_n$ for $a_n(x)$.
For all
$n_k\leq n< n_{k}+\nu(k)$, by the definition of $\mu$ and the length of a cylinder \eqref{qn-2},
 we have
\begin{align}
\frac{\log \mu(I_n(a_1,\ldots, a_n))}{\log |I_{n}(a_1,\ldots, a_{n})|}
&\ge \frac{\log \mu(I_{n_k}(a_1,\ldots, a_{n_k}))}{\log |I_{n_k+\nu(k)-1}(a_1,\ldots, a_{n_k+\nu(k)-1})|}\nonumber\\
&\ge \frac{\displaystyle{\sum_{m=1}^k\sum_{i\in V_m}\log [(2i+1)^t-(2i)^t]}}{\log 2+2\displaystyle{\sum_{i=1}^{n_k+\nu(k)-1}\log (a_i+1)}}\label{f1}\\
&=\frac{\displaystyle{\sum_{m=1}^k\sum_{i\in V_m}\log [(2i+1)^t-(2i)^t]}}{\log 2+2\left[\displaystyle{\sum_{m=1}^k\sum_{i\in V_m}\log (a_i+1)+\sum_{m=1}^k\sum_{i\in W_m}\log (a_i+1)}\right]}. \nonumber\end{align}
Recall the choice of the partial quotients:  \begin{align*}
a_i+1\le \left\{
           \begin{array}{ll}
             (2i+1)^t, & \hbox{for $i\in V_m$;} \\
             a_{n_m}+\nu(m)\le (2n_m+1)^t+\nu(m), & \hbox{for $i\in W_m$.}
           \end{array}
         \right.
\end{align*} Note that there are $\nu(m)-1$ elements in $W_k$. One has
\[\begin{split}
&\frac{\log \mu(I_n(a_1,\ldots, a_n))}{\log |I_{n}(a_1,\ldots, a_{n})|}\\
\ge &\frac{\displaystyle{\sum_{m=1}^k\sum_{i\in V_m    }\log [(2i+1)^t-(2i)^t]}}{\log 2+ 2\left[\displaystyle{\sum_{m=1}^k
\sum_{i\in V_m
    }t\log(2i+1)+\sum_{m=1}^k(\nu(m)-1)\log \Big((2n_m+1)^t+\nu(m)\Big)}\right]}\\
=&A_k.
\end{split}\]

For $n_k+\nu(k)\le n\le n_{k+1}$, by the definition of $\mu$ and the length of a cylinder \eqref{qn-2},
 we have
\begin{align*}
\frac{\log \mu(I_n(a_1,\ldots, a_n))}{\log |I_{n}(a_1,\ldots, a_{n})|}
&\ge \frac{\displaystyle{\sum_{m=1}^k\sum_{i\in V_m}\log [(2i+1)^t-(2i)^t]+\sum_{i=\min V_{k+1} }^n \log [(2i+1)^t-(2i)^t]}}{\log 2+2\left[\displaystyle{\sum_{i=1}^{n_k+\nu(k)-1}\log (a_i+1)+\sum_{i=\min V_{k+1} }^n \log (a_i+1)}\right]}.\end{align*}
Recalling (\ref{f1}), it follows that
$$\frac{\log \mu(I_n(a_1,\ldots, a_n))}{\log |I_{n}(a_1,\ldots, a_{n})|}\ge \min\{A_k, B_k\}.$$
The proof of Lemma~\ref{ratio} is completed.
\end{proof}

\begin{lem}\label{l4}
For any $x\in \Lambda_t(F)$, $$
\lim_{n\to \infty}\frac{\log |I_n(x)|}{\log |I_{n+1}(x)|}=1.
$$
\end{lem}\begin{proof}
By the length of a cylinder \eqref{qn-2}, the claim follows from the fact $$
\lim_{n\to\infty}a_n(x)=\infty, \ {\text{and}}\ \lim_{n\to \infty}\frac{\log a_n(x)}{\log a_{n+1}(x)}=1.
$$
\end{proof}

\begin{lem}\label{l5}
Let $x$ be an irrational number. Assume that $a_n(x)\ge 2$ for some $n\in \mathbb{N}$. Let $r>0$ be such that $$
|I_{n+1}(x)|\le r< |I_n(x)|.
$$
Then the ball $B(x,r)$ can intersect at most 4 cylinders of order $n$.
\end{lem}\begin{proof}
By the monotonicity of $x$ with its partial quotients, the cylinder $I_n(a_1(x),\cdots,$ $a_n(x)-1)$ lies in one side of $x$ and $$
I_n(a_1(x),\cdots, a_n(x)+1), \ I_n(a_1(x),\cdots, a_n(x)+2)
$$ lie in the other side of $x$. Moreover, a direct calculation shows that $$
|I_n(x)|\le |I_n(a_1(x),\cdots, a_n(x)-1)|$$ and $$ |I_n(x)|\le |I_n(a_1(x),\cdots, a_n(x)+1)|+ |I_n(a_1(x),\cdots, a_n(x)+2)|.
$$ Thus $$
B(x,r)\subset \bigcup_{i=-1}^2I_n(a_1(x),\cdots, a_n(x)+i).
$$
\end{proof}
We are in position to complete the proof of Proposition~\ref{lower-prop}.
Let $x\in \Lambda_t(F)$ and $r\in(0,1)$. We write $a_n$ for $a_n(x)$. Take  $n\in\mathbb N$ such that
\[|I_{n+1}(a_1,\ldots,a_{n+1})|\leq r<|I_n(a_1,\ldots,a_n )|.\]
By Lemma \ref{l5}, we have $\mu(B(x,r))\leq 4\mu(I_n(a_1,\ldots,a_n ))$, and
\[\frac{\log\mu(B(x,r))}{\log r}\geq \frac{\log[4\mu(I_n(a_1,\ldots,a_n ))]}{\log|I_{n+1}(a_1,\ldots,a_{n+1} )|}.\]
By Lemmas~\ref{AkBk}, \ref{ratio} and \ref{l4}, we obtain
\[\liminf_{r\to 0}\frac{\log\mu(B(x,r))}{\log r}\ge \frac{t-1}{2t(1+\alpha)}.\]
Lemma~\ref{mass1} gives
$\dim_{\rm H}F\geq\dim_{\rm H} \Lambda_t(F)\ge (t-1)/(2t(1+\alpha)).$
Letting $t\to \infty$ we obtain $\dim_{\rm H} F\ge 1/(2(1+\alpha)).$  The proof of Proposition~\ref{lower-prop} is complete.
\end{proof}

\subsection{Proof of Theorem~\ref{thmE}}\label{pfthmE}
 If $\alpha\in(0,\infty)$, then
  Propositions~\ref{upper-prop} and \ref{lower-prop} yield
 $\dim_{\rm H} F=1/(2(1+\alpha))$.
If $\alpha=0$, then
Proposition~\ref{lower-prop} yields
 $\dim_{\rm H} F\geq1/2$.
The reverse inequality follows from
$F\subset J$ and $\dim_{\rm H}J=1/2$ from \cite{R85}.
 If $\alpha=\infty$, then
 Proposition~\ref{upper-prop} yields
$\dim_{\rm H} F=0$.
\qed

\subsection{Lower bounds for $G$-sets }\label{lower-G}
The next proposition implies the lower bound for the Hausdorff dimension of $G$-sets.
\begin{pro}
\label{lower-prop2} Let $(\sigma_n)_{n=1}^\infty$ be a strictly increasing sequence of positive integers with $\sigma_{n+1}\ge \sigma_n+n$ for all $n\ge 1$ and $\lim_{n\to \infty}\frac{\sigma_n-\sigma_{n-1}}{n}=\beta$.
Then we have
\[\dim_{\rm H}  G((\sigma_n)_{n=1}^\infty)\geq\begin{cases}\vspace{1.5mm}\displaystyle{\frac{\beta-1}{2\beta}}&\text{ if }\beta<\infty,\\
\displaystyle{\ \  \frac{1}{2}}&\text{ otherwise}.\end{cases}
\]
\end{pro}

\begin{proof}We proceed much in parallel to the proof of Proposition~\ref{lower-prop} with similar notation.
When $\beta=1$, there is nothing to prove since we have proved $\dim_H G((\sigma_n)_{n=1}^\infty)\le 0$. Now assume $\beta>1$.

%
To
avoid triple subscripts, put $\sigma(k)=\sigma_{k}$.
Set $\sigma(0)=1$ for convenience.
  For each $k\in\mathbb N$, define sets $V_k$, $W_k$ of positive integers by
  \[V_k=\{\sigma(k-1)+{k-1},\ldots,\sigma(k)\}\ \text{ and }\ W_k=\{\sigma(k)+1,\ldots,\sigma(k)+k-1\}.\]
Then $\{V_1,W_1,V_2,W_2,\ldots\}$ is a partition of $\mathbb N$ into bounded intervals.
For an integer $t\ge 2$, define
\[\Lambda_t(G)=
\left\{
\begin{tabular}{l}
\!\!\!$x\in (0,1)\setminus\mathbb Q\colon$\!\!\!
$(2n)^t\le a_n(x)< (2n+1)^t$ for $n\in \bigcup_{k=1}^\infty V_k$ and \!\!\!\\ \ \ \ \ \ \ \ \ \ \ \ \ \ \ \ \ \ \
$a_{n}(x)=a_{\sigma(k) }(x)+n-\sigma(k)$ for $n\in W_k$, $k\in \mathbb N$\!\!\end{tabular}
\right\}.
\]
By Lemma~\ref{abstract}(a), there exists
a Borel probability measure $\mu$ on $[0,1]$
such that $\mu(\Lambda_t(G))=1$, and
for all $n\in\mathbb N$ and all $(a_1,\ldots,a_n)\in\mathbb N^n$ with
$I_n(a_1,\ldots,a_n)\cap \Lambda_t(G)\neq\emptyset$,
we have
\[\mu(I_n(a_1,\ldots, a_n))=\begin{cases}\displaystyle{\prod_{i=1}^n
\frac{1}{    (2i+1)^t-(2i)^t }  }&\text{ for }n\in V_1,\\
\displaystyle{\prod_{m=1}^k
\prod_{i\in V_m}\frac{1}{ (2i+1)^t-(2i)^t  }}&\text{ for }
n\in W_k,\
k\geq1,\\
\displaystyle{\begin{split}&\prod_{m=1}^k\prod_{i\in V_m}\frac{1}{ (2i+1)^t-(2i)^t }\\
&\times\prod_{i=\min V_{k+1} }^{n}\frac{1}{ (2i+1)^t-(2i)^t }\end{split}}
&\text{ for }
n\in V_{k+1},\ k\geq1.
\end{cases}\]
By Lemma~\ref{abstract}(b), $\Lambda_t(G)$ is contained in $G$.

We estimate the Hausdorff dimension of $\Lambda_t(G)$ from below.
For each $k\in\mathbb N$, define two positive numbers

    \[A_k=
    \frac{\displaystyle{\sum_{m=1}^{k }\sum_{i\in V_m    }\log [(2i+1)^t-(2i)^t]}}{\log 2+2\left[\displaystyle{
 \sum_{m=1}^{k }
\sum_{i\in V_m
    }t\log(2i+1)+\sum_{m=1}^k(m-1)\log [(2\sigma(m)+1)^t+m] }\right]} \]

    and
 \[B_k=\inf_{n\in V_{k+1}}\frac{\displaystyle{\sum_{i=\min V_{k+1} }^n \log [(2i+1)^t-(2i)^t]}}{\displaystyle{2 \sum_{i=\min V_{k+1} }^{n} t\log (2i+1)  }}.\]
\begin{lem}\label{CkDk}
For all $\beta\in[1,\infty]$,
\[\begin{split}\liminf_{k\to\infty} A_k&\geq\begin{cases}\vspace{1mm} \displaystyle{\frac{(t-1)(\beta-1)}{2t\beta}}\ &\text{ if }\beta<\infty,\\
 \displaystyle{\ \ \frac{1}{2}}\ &\text{ otherwise,}\end{cases}\ \ {\text{and}}\
  \lim_{k\to\infty}B_k= \frac{t-1}{2t}. \end{split}\]\end{lem}
\begin{proof}
For two sequences $(a_m)_{m=1}^\infty$, $(b_m)_{m=1}^\infty$ of positive reals, let us write $a_m\sim b_m$ if $\lim_{m\to\infty}a_m/b_m=1$, and $a_m\lesssim b_m$ if $\liminf_{m\to\infty} a_m/b_m\leq1$. Since $\sigma(m)\ge m$ and $t>1$ we have
$(2\sigma(m)+1)^t+ m\sim (2\sigma(m)+1)^t$. This yields
\[\begin{split}
    &\frac{\displaystyle{\sum_{i\in V_m    }\log [(2i+1)^t-(2i)^t]}}{\displaystyle{
\sum_{i\in V_m
    }t\log(2i+1)+(m-1)\log [(2\sigma(m)+1)^t+m] }}\\
    \sim&\frac{\displaystyle{\sum_{i\in V_m    }(t-1)\log (2i)}}{\displaystyle{
\sum_{i\in V_m
    }t\log(2i+1)+t(m-1)\log \sigma(m)}}\sim\frac{t-1}{t}\cdot\frac{1}{\displaystyle{
1+\frac{m\log \sigma(m) }{\sum_{i\in V_m    }\log i}}}.\end{split}\]

Note at first that $$
\frac{m\log \sigma(m)}{\sum_{i\in V_m}\log i}\le \frac{m\log \sigma(m)}{(\sigma(m)-\sigma(m-1)-m)\log \sigma(m-1)}.
$$ When $\beta<\infty$, it follows that $$
\frac{\sigma(m)}{m^2}\sim \beta/2,\ {\text{and thus}}\ \ \frac{\log \sigma(m)}{\log \sigma(m-1)}\sim 1.
$$ Consequently, $$
\liminf_{k\to \infty}A_k\ge \frac{t-1}{2t}\cdot \frac{1}{1+\frac{1}{\beta-1}}=\frac{(t-1)(\beta-1)}{2t\beta}.
$$
Now consider the case when $\beta=\infty$. Note at first that $$
\sigma(m)-\sigma(m-1)\ge m\Longrightarrow \sigma(m)\ge \frac{m^2}{4}.
$$ Secondly, if $\sigma(m)\ge 4\sigma(m-1)\ge 2m$, one has
$$\frac{m\log \sigma(m)}{(\sigma(m)-\sigma(m-1)-m)\log \sigma(m-1)}\le \frac{4m\log \sigma(m)}{\sigma(m)}\le \frac{4\sqrt{4\sigma(m)}\log \sigma(m)}{\sigma(m)};$$
if $\sigma(m)\le 4\sigma(m-1)$, $$\frac{m\log \sigma(m)}{(\sigma(m)-\sigma(m-1)-m)\log \sigma(m-1)}\le \frac{m\log 4+m\log \sigma(m-1)}{(\sigma(m)-\sigma(m-1)-m)\log \sigma(m-1)},$$ both of which tends to $0$ as $m\to\infty$. Consequently, $$
\liminf_{k\to \infty}A_k\ge \frac{t-1}{2t}.
$$
The convergence $B_k\to
 (t-1)/(2t)$ is obvious from the definition of $\Lambda_t(G)$.
\end{proof}
The remaining argument is absolutely the same as the proof the $F$-sets from Lemma \ref{ratio} on, so we omit the details.
%
\end{proof}

\subsection{Proof of Theorem~\ref{thmG}}
Let $(\sigma_n)_{n=1}^\infty$ be a strictly increasing sequence of positive integers with $\lim_{n\to \infty}\frac{\sigma_n-\sigma_{n-1}}{n}=\beta\ge 1$.
If $\beta\in[1,\infty)$, then Propositions~\ref{upper-prop-new} and \ref{lower-prop2} yield
$\dim_{\rm H}G=(\beta-1)/(2\beta)$.
If $\beta=\infty$, then Proposition~\ref{lower-prop2}
 yields $\dim_{\rm H}G\geq1/2$.
 The reverse inequality follows from
 $G\subset J$
 and $\dim_{\rm H}J=1/2$ from \cite{R85}.\qed

\subsection*{Acknowledgments}
This research was supported by National Key R$\&$D Program of China (No.2024YFA1013700), JSPS KAKENHI 25K17282, 25K21999, and NSFC 12331005.

\end{document}